\documentclass[11pt,a4paper]{amsart}
\usepackage{amssymb,url,upref,verbatim}
\usepackage[T1]{fontenc}
\usepackage{mathrsfs}\let\mathcal\mathscr
\makeatletter
\@namedef{subjclassname@2010}{%
\textup{2010} Mathematics Subject Classification}
\makeatother

\def\sha{\mathcal{A}}
\def\shb{\mathcal{B}}

\def\shd{\mathcal{D}}

\def\shf{\mathcal{F}}

\def\shh{\mathcal{H}}
\def\shl{\mathcal{L}}
\def\shm{\mathcal{M}}
\def\sho{\mathcal{O}}
\def\shn{\mathcal{N}}
\def\shr{\mathcal{R}}
\def\sht{\mathcal{T}}
\def\shv{\mathcal{V}}
\def\shx{\mathcal{X}}

\newcommand{\C}{\mathbb{C}}

\newcommand{\R}{\mathbb{R}}

\newcommand{\Z}{\mathbb{Z}}

\newcommand{\bD}{\boldsymbol{D}}
\newcommand{\Rhom}{R\shh\!om}
\newcommand{\rb}{\mathrm{b}}
\newcommand{\lc}{\mathrm{lc}}
\newcommand{\coh}{\mathrm{coh}}
\newcommand{\hol}{\mathrm{hol}}
\newcommand{\lccoh}{\mathrm{lc\,coh}}
\newcommand{\opp}{\mathrm{opp}}

\newcommand{\cc}{{\C\textup{-c}}}
\newcommand{\rc}{{\R\textup{-c}}}
\newcommand{\wcc}{{\textup{w-}\C\textup{-c}}}
\newcommand{\wrc}{{\textup{w-}\R\textup{-c}}}

\DeclareMathOperator{\Char}{Char}
\DeclareMathOperator{\codim}{codim}
\DeclareMathOperator{\pD}{{}^\mathrm{p}\mathsf{D}}
\DeclareMathOperator{\rD}{\mathsf{D}}
\DeclareMathOperator{\DR}{DR}
\DeclareMathOperator{\pDR}{{}^\mathrm{p}DR}
\DeclareMathOperator{\Hom}{Hom}
\DeclareMathOperator{\id}{id}
\DeclareMathOperator{\Sol}{Sol}
\DeclareMathOperator{\pSol}{{}^\mathrm{p}Sol}
\DeclareMathOperator{\supp}{Supp}

\let\ov\overline
\let\epsilon\varepsilon

\let\setminus\smallsetminus
\let\leq\leqslant
\let\geq\geqslant

\newcommand{\sbullet}{{\scriptscriptstyle\bullet}}

\theoremstyle{plain}
\newtheorem{theorem}{Theorem}[section]
\newtheorem{proposition}[theorem]{Proposition}
\newtheorem{lemma}[theorem]{Lemma}
\newtheorem{corollary}[theorem]{Corollary}

\theoremstyle{definition}
\newtheorem{definition}[theorem]{Definition}
\newtheorem{notation}[theorem]{Notation}
\newtheorem{example}[theorem]{Example}

\newtheorem{remark}[theorem]{Remark}

\newtheorem{properties}[theorem]{Properties}

\newcommand{\RedefinitSymbole}[1]{%
\expandafter\let\csname old\string#1\endcsname=#1
\let#1=\relax
\newcommand{#1}{\csname old\string#1\endcsname\,}%
}
\RedefinitSymbole{\forall} \RedefinitSymbole{\exists}

\begin{document}
\author{Teresa Monteiro Fernandes}
\address[T. Monteiro Fernandes]{Centro de Matem\'atica e Aplica\c c\~oes Fundamentais da Universidade
de Lisboa\\ Complexo 2\\ 2 Avenida Prof. Gama Pinto\\ 1699 Lisboa\\
Portugal}
\email{tmf@ptmat.fc.ul.pt}

\author{Claude Sabbah}
\address[C. Sabbah]{Centre de Math\'ematiques Laurent Schwartz \\
UMR CNRS 7640\\
\'Ecole Polytechnique\\91128 Palaiseau cedex\\ France}
\email{sabbah@math.polytechnique.fr}

\thanks{The research of T.M.F. was partially supported
by Supported by Funda\c c\~ao para a Ci\^encia e a Tecnologia, PEst OE/MAT/UI0209/2011.}
\thanks{The research of C.S. was supported by the grant ANR-08-BLAN-0317-01 of the Agence nationale de la recherche.}

\title{On the de Rham complex of mixed~twistor~$\mathcal{D}$-modules}
\date{11/02/2012}
\begin{abstract}
Given a complex manifold $S$, we introduce for each complex manifold $X$ a t-structure on the bounded derived category of $\mathbb{C}$\nobreakdash-constructible complexes of $\mathcal{O}_S$-modules on $X\times S$. We prove that the de~Rham complex of a holonomic $\mathcal{D}_{X\times S/S}$-module which is $\mathcal{O}_S$-flat as well as its dual object is perverse relatively to this t-structure. This result applies to mixed twistor $\mathcal{D}$-modules.
\end{abstract}



\maketitle
\renewcommand{\baselinestretch}{1.1}\normalfont

\section{Introduction}
Given a vector bundle $V$ of rank $d\geq1$ with an integrable connection $\nabla:V\to \Omega^1_X\otimes V$ on a complex manifold $X$ of complex dimension $n$, the sheaf of horizontal sections $V^\nabla=\ker\nabla$ is a locally constant sheaf of $d$\nobreakdash-dimensional $\C$-vector spaces, and is the only nonzero cohomology sheaf of the de~Rham complex $\DR_X(V,\nabla)=(\Omega_X^\sbullet\otimes V,\nabla)$. Assume moreover that $(V,\nabla)$ is equipped with a harmonic metric in the sense of \cite[p.~16]{Simpson92}. The twistor construction of \cite{Simpson97} produces then a holomorphic bundle $\shv$ on the product space $\shx=X\times\C$, where the factor $\C$ has coordinate $z$, together with a holomorphic flat $z$-connection. By restricting to $\shx^*:=X\times\C^*$, giving such a $z$-connection on $\shv^*:=\shv_{|\shx^*}$ is equivalent to giving a flat relative connection~$\nabla$ with respect to the projection $p:\shx^*\to\C^*$. Similarly, the relative de~Rham complex $\DR_{\shx^*/\C^*}(\shv^*,\nabla)$ has cohomology in degree zero at most, and $(\shv^*)^\nabla:=\ker\nabla$ is a locally constant sheaf of locally free $p^{-1}\sho_{\C^*}$\nobreakdash-modules of rank $d$.

Holonomic $\shd_X$-modules generalize the notion of a holomorphic bundle with flat connection to objects having (possibly wild) singularities, and a well-known theorem of Kashiwara \cite{Kashiwara75} shows that the solution complex of such a holonomic $\shd_X$-module has $\C$-constructible cohomology, from which one can deduce that the de~Rham complex is of the same kind and more precisely that both are $\C$-perverse sheaves on $X$ up to a shift by $\dim X$.

The notion of a holonomic $\shd_X$-module with a harmonic metric has been formalized in \cite{Sabbah05} and \cite{Mochizuki07} under the name of pure twistor $\shd$-module (this generalizes holonomic $\shd_X$-modules with regular singularities), and then in \cite{Sabbah09} and \cite{Mochizuki08} under the name of wild twistor $\shd$-modules (this takes into account arbitrary irregular singularities). More recently, Mochizuki \cite{Mochizuki11} has fully developed the notion of a mixed (possibly wild) twistor $\shd$-module. When restricted to $\shx^*$, such an object contains in its definition two holonomic $\shd_{\shx^*/\C^*}$-modules, and we say that both underlie a mixed twistor $\shd$-module

The main result of this article concerns the de~Rham complex and the solution complex of such objects.

\begin{theorem}\label{th:main}
The de~Rham complex and the solution complex of a $\shd_{\shx^*/\C^*}$-module underlying a mixed twistor $\shd$-module are perverse sheaves of $p^{-1}\sho_{\C^*}$-modules (up to a shift by $\dim X$).
\end{theorem}

In Section \ref{sec:constructibility}, we define the notion of relative constructibility and perversity. This applies to the more general setting where $p:\shx^*\to\C^*$ is replaced by a projection $p_X:\shx=X\times S\to S$, where $S$ is any complex manifold. We usually set $p=p_X$ when~$X$ is fixed. On the other hand, we call \emph{holonomic} any coherent $\shd_{X\times S/S}$-module whose relative characteristic variety in $T^*(X\times\nobreak S/S)=(T^*X)\times S$ is contained in a variety $\Lambda\times S$, where $\Lambda$ is a conic Lagrangian variety in $T^*X$. We say that a $\shd_{X\times S/S}$-module is \emph{strict} if it is $p^{-1}\sho_S$-flat.

\begin{theorem}\label{th:main2}
The de~Rham complex and the solution complex of a strict holonomic $\shd_{X\times S/S}$-module whose dual is also strict are perverse sheaves of $p^{-1}\sho_S$-modules (up to a shift by $\dim X$).
\end{theorem}

A $\shd_{\shx^*/\C^*}$-module $\shm$ underlying a mixed twistor $\shd$-module is strict and holonomic (see~\cite{Mochizuki11}). Moreover, Mochizuki has defined a duality functor on the category of mixed twistor $\shd$-modules, proving in particular that the dual of $\shm$ as a $\shd_{\shx^*/\C^*}$-module is also strict holonomic. Therefore, these results together with Theorem \ref{th:main2} imply Theorem \ref{th:main}.

Note that, while our definition of perverse objects in the bounded derived category $\rD^\rb(p^{-1}\sho_S)$ intends to supply a notion of holomorphic family of perverse sheaves, we are not able, in the case of twistor $\shd$-modules, to extend this notion to the case when the parameter $z\in\C^*=S$ also achieves the value zero, and to define a perversity property in the Dolbeault setting of \cite{Simpson92} for the associated Higgs module.

\section{Relative constructibility in the case of a projection}\label{sec:constructibility}

We keep the setting as above, but $X$ is only assumed to be a real analytic manifold. Given a real analytic map $f:Y\to X$ between real analytic manifolds, we will denote by $f_S$ (or $f$ if the context is clear) the map $f\times\id_S:Y\times S\to X\times S$.

\subsection{Sheaves of $\C$-vector spaces and of $p^{-1}\sho_S$-modules}\label{S:0}

Let $f:Y\!\to\! X$ be such a map. There are functors $f^{-1},f^!,Rf_*,Rf_!$ between $\rD^\rb(\C_{X\times S})$ and $\rD^\rb(\C_{Y\times S})$, and functors $f_S^{-1},f_S^!,Rf_{S,*},Rf_{S,!}$ between $\rD^\rb(p_X^{-1}\sho_S)$ and $\rD^\rb(p_Y^{-1}\sho_S)$. These functors correspond pairwise through the forgetful functor $\rD^\rb(p_X^{-1}\sho_S)\to \rD^\rb(\C_{X\times S})$. Indeed, this is clear except for $f_S^!$ and $f^!$. To check it, one decomposes $f$ as a closed immersion and a projection. In the first case, the compatibility follows from the fact that both are equal to $f^{-1}R\Gamma_{f(X)}$ (see \cite[Prop.\,3.1.12]{K-S90}) and for the case of a projection one uses \cite[Prop.\,3.1.11 \& 3.3.2]{K-S90}. We note also that the Poincaré-Verdier duality theorem \cite[Prop.\,3.1.10]{K-S90} holds on $\rD^\rb(p^{-1}\sho_S)$ (see \cite[Rem.\,3.1.6(i)]{K-S90}). From now on, we will write $f^{-1}$, etc. instead of $f_S^{-1}$, etc.

The ring $p_X^{-1}\sho_S$ is Noetherian, hence coherent (see \cite[Prop.\,A.14]{Kashiwara03}). For each $s_o\in S$ let us denote by $\mathfrak{m}_{s_o}$ the ideal of sections of $\sho_S$ vanishing at~$s_o$ and by $i^\star_{s_o}$ the functor
\begin{align*}
\text{Mod}(p^{-1}_X \sho_{S})& \longmapsto \text{Mod}(\C_X)\\
F&\longmapsto F\otimes _{p_X^{-1}\sho_{S}} p_X^{-1}(\sho_{S}/\mathfrak{m}_{s_o}).
\end{align*}
This functor will be useful for getting properties of $\rD^\rb(p_X^{-1}\sho_S)$ from well-known properties of $\rD^\rb(\C_X)$.

\begin{proposition}\label{C:CC}
Let $F$ and $F'$ belong to $\rD^\rb(p_X^{-1}\sho_{S})$. Then, for each $s_o\in S$ there is a well-defined natural morphism
\[
Li^*_{s_o}(\Rhom_{p^{-1}(\sho_{S})}(F, F')) \to \Rhom_{\C_{X}}(Li^*_{s_o}(F), Li^*_{s_o}(F'))
\]
which is an isomorphism in $\rD^\rb(\C_X)$.
\end{proposition}

\begin{proof}
Let us fix $s_o\in S$. The existence of the morphism follows from \cite [(A.10)]{Kashiwara03}. Moreover, since $p^{-1}_X\sho_S$ is a coherent ring as remarked above and $p_X^{-1}(\sho_{S}/\mathfrak{m}_{s_o})$ is $p^{-1}_X\sho_S$-coherent, we can apply the argument given after (A.10) in loc.\,cit.\ to show that it is an isomorphism.
\end{proof}

\begin{proposition}\label{P:91}
Let $F$ and $F'$ belong to $\rD^\rb(p^{-1}_X\sho_S)$ and let $\phi: F\to F'$ be a morphism. Assume the following conditions:
\begin{enumerate}
\item
for all $j\in\Z$ and $(x,s)\in X\times S$, $\shh^j(F)_{(x,s)}$ and $\shh^j(F')_{(x,s)}$ are of finite type over $\sho_{S,s}$,

\item
for all $s_o\in S$, the natural morphism
\[
Li^*_{s_o}(\phi): Li^*_{s_o}(F)\to Li^*_{s_o}(F')
\]
is an isomorphism in $\rD^\rb(\C_X)$.
\end{enumerate}

Then $\phi$ is an isomorphism.
\end{proposition}

\begin{proof}
It is enough to prove that the mapping cone of $\phi$ is quasi-isomorphic to $0$. So we are led to proving that for $F\in \rD^\rb(p^{-1}\sho_S)$, if $\shh^j(F)_{(x,s)}$ are of finite type over $\sho_{S,s}$ for all $(x,s)\in X\times\nobreak S$, and $Li^*_{s_o}(F)$ is quasi-isomorphic to $0$ for each $s_o\in S$, then $F$ is quasi-isomorphic to $0$.

We may assume that $S$ is an open subset of $\C^n$ with coordinates $s^1,\dots,s^n$ and we will argue by induction on $n$. Assume $n=1$. For such an $F$, for each $s_o\in S$ and any $j\in\Z$ the morphism $(s^1-s^1_o):\shh^j(F)\to \shh^j(F)$ is an isomorphism, hence $\shh^j(F)/(s^1-s^1_o)\shh^j(F)=0$ and by Nakayama's Lemma, for any $x\in X$, $\shh^j(F)_{(x,s^1_o)}=0$ and the result follows. For $n\geq2$, the cone~$F'$ of the morphism $(s^n-s^n_o):F\to F$ also satisfies $Li^*_{s'_o}F'=0$ for any $s'_o=(s^1_o,\dots,s^{n-1}_o)$, hence is zero by induction, so we can argue as in the case $n=1$.
\end{proof}

\subsection{$S$-locally constant sheaves}
We say that a sheaf $F$ of $\C$-vector spaces (resp.~$p_X^{-1}\sho_S$-modules) on $X\times S$ is \emph{$S$-locally constant} if, for each point $(x,s)\in X\times S$, there exists a neighbourhood $U=V_x\times T_s$ of $(x,s)$ and a sheaf $G^{(x,s)}$ of $\C$-vector spaces (resp.~$\sho_S$-modules) on $T_s$, such that $F_{|U}\simeq p_U^{-1}G^{(x,s)}$. The category of $S$-locally constant sheaves is an abelian full subcategory of that of sheaves of $\C_{X\times S}$-vector spaces (resp.~$p^{-1}\sho_S$-modules), which is stable by extensions in the respective categories, by $\shh\!om$ and tensor products. Moreover, if $\pi:Y\times\nobreak X\times\nobreak S\to Y\times S$ is the projection, with $X$ contractible, then, if $F'$ is $S$-locally constant on $Y\times X\times S$,
\begin{itemize}
\item
$\pi_*F'$ is $S$-locally constant on $Y\times S$,
\item
$R^k\pi_*F'=0$ if $k>0$,
\item
$F'\simeq\pi^{-1}\pi_*F'$.
\end{itemize}
Applying this to $Y=\{\textup{pt}\}$, we find that, if $F$ is $S$-locally constant, then for each $x\in X$ there exists a connected neighbourhood $V_x$ of $x$ and a $\C_S$-module (resp.~$\sho_S$-module) $G^{(x)}$ such that $F=p_{V_x}^{-1}G^{(x)}$, and one has $G^{(x)}=p_{V_x,*}F_{|V_x\times S}=F_{|\{x\}\times S}$. We shall also denote by $\rD^\rb_\lc(p_X^{-1}\C_{S})$ (resp.~$\rD^\rb_\lc(p_X^{-1}\sho_{S})$) the bounded triangulated category whose objects are the complexes having $S$-locally constant cohomology sheaves. Similarly, for such a complex $F$ we have $F_{|V_x\times S}\simeq p_{V_x}^{-1}Rp_{V_x,*}F_{|V_x\times S}\simeq p_{V_x}^{-1}F_{|\{x\}\times S}$.

We conclude from the previous remarks, by using the natural forgetful functor $\rD^\rb(p_X^{-1}\sho_{S})\to \rD^\rb(\C_{X\times S})$:

\begin{lemma}\label{L:0}\mbox{}
\begin{enumerate}
\item\label{L:0a}
An object $F$ of $\rD^\rb(p_X^{-1}\sho_{S})$ belongs to $\rD^\rb_\lc(p_X^{-1}\sho_{S})$ if and only if, when regarded as an object of $\rD^\rb(\C_{X\times S})$, it belongs to $\rD^\rb_\lc(p_X^{-1}\C_{S})$.
\item\label{L:0b}
For any object $F$ of $\rD^\rb_\lc(p_X^{-1}\sho_{S})$ and for any $s_o\in S$, $Li^*_{s_o}F$ belongs to $\rD^\rb_\lc(\C_X)$.
\end{enumerate}
\end{lemma}

\subsection{$S$-weakly $\R$-constructible sheaves}
As long as the manifold $X$ is fixed, we shall write $p$ instead of $p_X$.

\begin{definition}\label{D:3}
Let $F\in \rD^\rb(\C_{X\times S})$ (resp.~$F\in \rD^\rb(p^{-1}\sho_{S})$). We shall say that $F$ is $S$-weakly $\R$-constructible if there exists a subanalytic $\mu$\nobreakdash-strati\-fi\-cation $(X_\alpha)$ of $X$ (see \cite[Def.~8.3.19]{K-S90}) such that, for all $j\in \Z$, $\shh^j(F)|_{X_\alpha\times S}$ is $S$-locally constant.
\end{definition}

This condition is independent of the choice of the $\mu$-stratification and characterizes a full triangulated subcategory $\rD^\rb_\wrc(p^{-1}\C_S)$ (resp.\ $\rD^\rb_\wrc(p^{-1}\sho_{S})$) of $\rD^\rb(\C_{X\times S})$ (resp.~$\rD^\rb(p^{-1}\sho_{S})$). Due to Lemma~\ref{L:0}, an object~$F$ of $\rD^\rb(p^{-1}\sho_{S})$ is in $\rD^\rb_\wrc(p^{-1}\sho_{S})$ if and only if it belongs to $\rD^\rb_\wrc(p^{-1}\C_S)$ when considered as an object of $\rD^\rb(\C_{X\times S})$. By mimicking for $\rD^\rb_\wrc(p^{-1}\C_S)$ the proof of \cite[Prop.\,8.4.1]{K-S90} and according to the previous remark for $\rD^\rb_\wrc(p^{-1}\sho_{S})$, we obtain:

\begin{proposition}\label{P:6}
Let $F$ be $S$-weakly $\R$-constructible on $X$ and let $X=\bigsqcup_\alpha X_\alpha$ be a $\mu$-stratification of $X$ adapted to $F$. Then the following conditions are equivalent:
\begin{enumerate}
\item\label{P:6-1}
for all $j\in\Z$ and for all $\alpha$, $\shh^j(F)|_{X_\alpha\times S}$ is $S$-locally constant.
\item\label{P:6-2}
$SS(F)\subset (\bigsqcup_\alpha T^*_{X_\alpha} X)\times T^* S$.
\item\label{P:6-3}
There exists a closed conic subanalytic Lagrangian subset $\Lambda$ of $T^*X$ such that $SS(F)\subset \Lambda\times T^*S$.
\end{enumerate}
\end{proposition}

\begin{proposition}\label{C:CCC}
Let $F\in \rD^\rb_\wrc(p_{X}^{-1}\sho_{S})$ and let $s_o\in S$. Then $Li_{s_o}^*(F)\in \rD^\rb_\wrc(\C_X)$.
\end{proposition}

\begin{proof}
Let $i_\alpha:X_\alpha\hookrightarrow X$ denote the locally closed inclusion of a stratum of an adapted stratification $(X_\alpha)$. It is enough to observe that, for each $\alpha$, we have $i^{-1}_\alpha Li^*_{s_o}(F)\simeq Li^*_{s_o}(i_\alpha^{-1}F)$, and to apply Lemma \ref{L:0}\eqref{L:0b}.
\end{proof}

Let now $Y$ be another real analytic manifold and consider a real analytic map $f: Y\to X$. The following statements for objects of $\rD^\rb_\wrc(p^{-1}\C_S)$ are easily deduced from Proposition \ref{P:6} similarly to the absolute case treated in~\cite{K-S90}, as consequences of Theorem $8.3.17$, Proposition $8.3.11$, Corollary $6.4.4$ and Proposition $5.4.4$ of loc.\,cit.\ In order to get the same statements for objects of $\rD^\rb_\wrc(p^{-1}\sho_S)$, one uses Lemma \ref{L:0}\eqref{L:0a} together with \S\ref{S:0}. We will not distinguish between $f$ and $f_S$.

\begin{proposition}\label{P:01}\mbox{}
\begin{enumerate}
\item
If $F$ is $S$-weakly $\R$-constructible on $X$, then so are $f^{-1}(F)$ and $f^!(F)$.

\item
Assume that $F'$ is $S$-weakly $\R$-constructible on $Y$ and that $f$ is proper on $\supp(F')$. Then $Rf_*(F')$ is $S$-weakly $\R$-constructible on $X$.
\end{enumerate}
\end{proposition}

Given a closed subanalytic subset $Y\subset X$, we will denote by $i:Y\times S\hookrightarrow X\times S$ the closed inclusion and by $j$ the complementary open inclusion.

\begin{corollary}\label{P:01b}
Assume that $F^*$ is $S$-weakly $\R$-constructible on $X\setminus Y$. Then the objects $Rj_!F^*$ and $Rj_*F^*$ are also $S$-weakly $\R$-constructible on $X$.
\end{corollary}

\begin{proof}
The statement for $Rj_!F^*$ is obvious. Then Proposition \ref{P:01} implies that $i^!Rj_!F^*$ is $S$-weakly $\R$-constructible. Conclude by using the distinguished triangle
\[
Ri_*i^!Rj_!F^*\to Rj_!F^*\to Rj_*F^*\xrightarrow{~+1}
\]
and the $S$-weak $\R$-constructibility of the first two terms.
\end{proof}

\begin{proposition}\label{C:2.6}
An object $F\in \rD^\rb(\C_{X\times S})$ (resp.~$F\in \rD^\rb(p^{-1}(\sho_S))$) is $S$-weakly $\R$-constructible with respect to a $\mu$-stratification $(X_\alpha)$ if and only if, for each $\alpha$, $i_\alpha^!F$ has $S$-locally constant cohomology on $X_\alpha$.
\end{proposition}

\begin{proof}
Assume that $F$ is $S$-weakly $\R$-constructible with respect to a $\mu$\nobreakdash-strati\-fi\-cation $(X_\alpha)$ of $X$. Then $i_\alpha^!F$ has $S$-locally constant cohomology on $X_\alpha$. Indeed the estimation of the micro-support of \cite[Cor.\,6.4.4(ii)]{K-S90} implies that $SS(i_\alpha^!F)$ (like $SS(i_\alpha^*F)$) is contained in $T^*_{X_\alpha}{X_\alpha}\times T^*S$, so $i_\alpha^!F$ has locally constant cohomology on $X_\alpha$ for each $\alpha$, according to Proposition \ref{P:6}.

Conversely, if $i_\alpha^!F$ is locally constant for each $\alpha$, then $F$ is $S$-weakly $\R$-constructible. Indeed, we argue by induction and we denote by $X_k$ the union of strata of codimension $\leq k$ in $X$. Assume we have proved that $F_{|X_{k-1}\times S}$ is $S$-weakly $\R$-constructible with respect to the stratification $(X_\alpha)$ with $\codim X_\alpha\leq k-1$. We denote by $j_k:X_{k-1}\hookrightarrow X_k$ the open inclusion and by $i_k$ the complementary closed inclusion. According to Corollary \ref{P:01b}, $Rj_{k,*}j_k^{-1}F$ is $S$-weakly $\R$-constructible with respect to $(X_\alpha)_{|X_k}$. Now, by using the exact triangle $i_k^!F\to i_k^{-1}F\to i_k^{-1}Rj_{k,*}j_k^{-1}F\xrightarrow{+1}$, we conclude that $i_k^{-1}F$ is locally constant, hence $F_{|X_k\times S}$ is $S$-weakly $\R$-constructible.
\end{proof}

\begin{corollary}\label{C:2.6.1}
Let $F,F'\in \rD^\rb_\wrc(p_X^{-1}\sho_S)$. Then $\Rhom_{p_X^{-1}\sho_S}(F,F')$ also belongs to $\rD^\rb_\wrc(p_X^{-1}\sho_S)$.
\end{corollary}

\begin{proof}
In view of Proposition \ref{C:2.6}, it is sufficient to prove that for each $\alpha$, $i_\alpha^!\Rhom_{p_X^{-1}\sho_S}(F,F')$ belongs to $\rD^\rb_\lc(p_X^{-1}\sho_S)$. We have:
\[
i_\alpha^!\Rhom_{p^{-1}\sho_S}(F,F')\simeq
\Rhom_{p_\alpha^{-1}\sho_S}(i^{-1}_\alpha F,i_\alpha^!F').
\]
Since both $i^{-1}_\alpha F$ and $i_\alpha^!F'$ belong to $\rD^\rb_\lc(p_X^{-1}\sho_S)$, according to Proposition \ref{C:2.6}, we have locally on $X_\alpha$ isomorphisms $i^{-1}_\alpha F=p_\alpha^{-1}G_\alpha$ and $i_\alpha^!F'=p_\alpha^{-1}G'_\alpha=p_\alpha^!G'_\alpha[-\dim_\R X_\alpha]$ for some $\sho_S$-modules $G_\alpha$ and $G'_\alpha$. Then
\begin{align*}
\Rhom_{p_\alpha^{-1}\sho_S}(i^{-1}_\alpha F,i_\alpha^!F')&=\Rhom_{p_\alpha^{-1}\sho_S}(p^{-1}_\alpha G_\alpha,p_\alpha^!G'_\alpha[-\dim_\R X_\alpha])\\
&\simeq p_\alpha^!\Rhom_{\sho_S}(G_\alpha,G'_\alpha)[-\dim_\R X_\alpha]\\
&=p_\alpha^{-1}\Rhom_{\sho_S}(G_\alpha,G'_\alpha).\qedhere
\end{align*}
\end{proof}

The following lemma will be useful in the next section. Assume that $X=Y\times Z$ and that the $\mu$-stratification $(X_\alpha)$ of $X$ takes the form $X_\alpha=Y\times Z_\alpha$, where $(Z_\alpha)$ is a $\mu$-stratification of $Z$. We denote by $q:X\to Y$ the projection. Let $z_o\in Z$, let $U\ni z_o$ be a coordinate neighbourhood of $z_o$ in~$Z$ and, for each $\epsilon>0$ small enough, let $B_\epsilon\subset U$ be the open ball of radius~$\epsilon$ centered at $z_o$ and let $\ov B_\epsilon$ be the closed ball and $S_\epsilon$ its boundary. For the sake of simplicity, we denote by $q_\epsilon,q_{\ov\epsilon},q_{\partial\epsilon}$ the corresponding projections.

We set $Z^*=Z\setminus\{z_o\}$ and $X^*=Y\times Z^*$. We denote by $i:Y\times\{z_o\}\hookrightarrow Y\times Z$ and by $j:Y\times Z^*\hookrightarrow Y\times Z$ the complementary closed and open inclusions.

\begin{lemma}\label{L:2.7}
Let $F^*\in \rD^\rb_\wrc(p_{X^*}^{-1}\C_S)$ (resp.~$F^*\in \rD^\rb_\wrc(p_{X^*}^{-1}\sho_{S})$) be adapted to the previous stratification. Then there exists $\epsilon_o>0$ such that, for each $\epsilon\in(0,\epsilon_o)$, the natural morphisms
\[
Rq_{\partial\epsilon,*}F^*_{|Y\times S_\epsilon\times S}\longleftarrow Rq_{\ov\epsilon,*}Rj_*F^*\longrightarrow Rq_{\epsilon,*}Rj_*F^*\longrightarrow i^{-1}Rj_*F^*
\]
are isomorphisms.
\end{lemma}

\begin{proof}
We note that, according to Corollary \ref{P:01b}, $F:=Rj_*F^*$ is $S$\nobreakdash-weakly $\R$-constructible, and is adapted to the stratification $(Y\times Z_\alpha)$. On the other hand, according to \S\ref{S:0}, it is enough to consider the case where $F^*$ is an object of $\rD^\rb_\wrc(p_{X^*}^{-1}\C_S)$.

Let us start with the right morphisms. We can argue with any object $F\in \rD^\rb_\wrc(p_X^{-1}\C_S)$, not necessarily of the form $Rj_*F^*$. Recall that we have an adjunction morphism $q_{\epsilon}^{-1}Rq_{\epsilon,*}\to\id$ and thus $i^{-1}q_{\epsilon}^{-1}Rq_{\epsilon,*}\to i^{-1}$. Since $q_{\epsilon}\circ i=\id_{Y\times S}$, we get the second right morphism. The first one is the restriction morphism.

According to \cite[Prop.\,8.3.12 and 5.4.17]{K-S90}, there exists $\epsilon_o>0$ such that, for $\epsilon'<\epsilon$ in $(0,\epsilon_o)$, the restriction morphisms $Rq_{\ov\epsilon,*}F\to Rq_{\epsilon,*}F\to Rq_{\ov\epsilon',*}F\to Rq_{\epsilon',*}F$ are isomorphisms. In particular, the first right morphism is an isomorphism.

Let us take a $q$-soft representative of $F$, that we still denote by $F$. The inductive system $q_{\epsilon,*}F$ ($\epsilon\to0$) has limit $i^{-1}F$ and all morphisms of this system are quasi-isomorphisms. Hence the second right morphism is a quasi-isomorphism.

\begin{remark}\label{R:2.8}
A similar argument gives an isomorphism $i^!F\xrightarrow{\sim}Rq_{\epsilon,!}F$, by using \cite[Prop.\,5.4.17(c)]{K-S90}.
\end{remark}

For the left morphism, we take a $q$-soft representative of $F^*$ that we still denote by $F^*$. For $\epsilon_-<\epsilon<\epsilon_+<\epsilon_o$, we denote by $B_{\epsilon_-,\epsilon_+}$ the open set $B_{\epsilon_+}\setminus\ov B_{\epsilon_-}$ and by $q_{\epsilon_-,\epsilon_+}$ the corresponding projection. We have $q_{\partial\epsilon,*}F^*=\varinjlim_{|\epsilon_+-\epsilon_-|\to0}q_{\epsilon_-,\epsilon_+,*}F^*$. On the other hand, the morphisms of this inductive system are all quasi-isomorphisms, according to \cite[Prop.\,5.4.17]{K-S90}. Fixing $\epsilon'\in(\epsilon,\epsilon_o)$ we find a quasi-isomorphism $q_{\epsilon',*}F^*\to q_{\partial\epsilon,*}F^*$. On the other hand, from the first part we have $q_{\epsilon',*}F^*\xrightarrow{\sim}q_{\ov\epsilon,*}F^*$, hence the result.
\end{proof}

\subsection{$S$-coherent local systems and $S$-$\R$-constructible sheaves}

\begin{notation}\label{D:09}
We shall denote by $\rD^\rb_{\lccoh}(p_X^{-1}\sho_S)$ the full triangulated subcategory of $\rD^\rb_\lc(p_X^{-1}\sho_S)$ whose objects satisfy, locally on $X$, $F\simeq p_X^{-1}G$ with $G\in \rD^\rb_\coh(\sho_S))$. Equivalently, for each $x\in X$, $F_{|\{x\}\times S}\in \rD^\rb_\coh(\sho_S)$ (see~the remarks before Lemma \ref{L:0}).
\end{notation}

\begin{definition}\label{D:31}
Given $F\in \rD^\rb_\wrc(p_X^{-1}\sho_S)$, we say that $F$ is $\R$\nobreakdash-construc\-ti\-ble if, for some $\mu$-stratification of $X$, $X=\bigsqcup_\alpha X_\alpha$, for all $j\in \Z$, $\shh^j(F)|_{X_\alpha\times S}\in \rD^\rb_\lccoh(p_{X_\alpha}^{-1}\sho_S)$. This condition characterizes a full triangulated subcategory of $\rD^\rb_\wrc(p_X^{-1}\sho_{S})$ which we denote by $\rD^\rb_\rc(p_X^{-1}\sho_{S})$.
\end{definition}

Similarly to Proposition \ref{C:CCC} we have:

\begin{proposition}\label{C:CCCC}
Let $F\in \rD^\rb_\rc(p_{X}^{-1}\sho_{S})$ and let $s_o\in S$. Then $Li_{s_o}^*(F)\in \rD^\rb_\rc(\C_X)$.
\end{proposition}

\begin{remark}\label{R:0}
An object of $\rD^\rb_\wrc(p_X^{-1}\sho_S)$ is in $\rD^\rb_\rc(p_X^{-1}\sho_S)$ if and only if, for any $x\in X$, $F_{|\{x\}\times S}$ belongs to $\rD^\rb_\coh(\sho_S)$.
\end{remark}

A straightforward adaptation of \cite[Prop.\,8.4.8]{K-S90} gives:

\begin{proposition}\label{P:07}
Let $f: Y\to X$ be a a morphism of manifolds and let $F\in \rD^\rb_\rc(p_Y^{-1}\sho_{S})$. Assume that $f_S$ is proper on $\supp(F)$. Then
\[
Rf_{S,*}F\in \rD^\rb_\rc(p_X^{-1}\sho_{S}).
\]
\end{proposition}

We can also characterize $\rD^\rb_\rc(p_X^{-1}\sho_S)$ as in Corollary \ref{C:2.6}.

\begin{corollary}\label{C:2.12}
An object $F\in \rD^\rb(p_X^{-1}\sho_S)$ is in $\rD^\rb_\rc(p_X^{-1}\sho_S)$ if and only if, for some subanalytic Whitney stratification $(X_\alpha)$ of $X$, the complexes $i_\alpha^!F$ belong to $\rD^\rb_{\lccoh}(p_\alpha^{-1}\sho_S)$.
\end{corollary}

\begin{proof}
Assume $F$ is in $\rD^\rb_\rc(p_X^{-1}\sho_S)$. We need to prove the coherence of $i_\alpha^!F$. We argue by induction as in Corollary \ref{C:2.6}, with the same notation. Since the question is local on $X_k$, by the Whitney property of the stratification $(X_\alpha)$ we can assume that $X_{k-1}=Z\times Y_k$ and there exists a Whitney stratification $(Z_\alpha)$ of $Z$ such that $X_\alpha=Z_\alpha\times Y_k$ for each $\alpha$ such that $X_\alpha\subset X_{k-1}$ (see e.g. \cite[\S1.4]{G-M88}). Proving that $i_k^!F$ is $p^{-1}\sho_S$-coherent is equivalent to proving that $i_k^{-1}Rj_{k,*}j_k^{-1}F$ is so, since we already know that $i_k^{-1}F$ is so. According to Lemma \ref{L:2.7}, $i_k^{-1}Rj_{k,*}j_k^{-1}F$ is computed as $Rq_{\partial\epsilon,*}j_k^{-1}F$, and since $q_{\partial\epsilon}$ is proper, we can apply Proposition \ref{P:07} to get the coherence.

Conversely, Corollary \ref{C:2.6} already implies that $F$ is an object of $\rD^\rb_\wrc(p_X^{-1}\sho_S)$. We argue then as above: since we know by assumption that $i_k^!F$ is coherent, it suffices to prove that $i_k^{-1}Rj_{k,*}j_k^{-1}F$ is so, and the previous argument applies.
\end{proof}

\subsection{$S$-weakly $\C$-constructible sheaves and $S$-$\C$-constructible sheaves}
Let now assume that $X$ is a complex analytic manifold.

\begin{definition}\label{D:001}\mbox{}
\begin{enumerate}
\item
Let $F\in \rD^\rb_\wrc(p_X^{-1}\C_{S})$ (resp.~$F\in \rD^\rb_\wrc(p_X^{-1}\sho_{S})$). We shall say that $F$ is $S$-weakly $\C$-constructible if $SS(F)$ is $\C^*$-conic. The corresponding categories are denoted by $\rD^\rb_\wcc(p_X^{-1}\C_{S})$ (resp.~$F\in \rD^\rb_\wcc(p_X^{-1}\sho_{S})$).

\item
If $F$ belongs to $\rD^\rb_\wcc(p_X^{-1}\sho_{S})$, we say that $F$ is $S$-$\C$-constructible if $F\in \rD^\rb_\rc(p_X^{-1}\sho_{S})$, and we denote by $\rD^\rb_\cc(p_X^{-1}\sho_{S})$ the corresponding category, which is full triangulated sub-category of $\rD^\rb(p_X^{-1}\sho_S)$.
\end{enumerate}
\end{definition}

The following properties are obtained in a straightforward way, by using \cite[Th.~8.5.5]{K-S90} in a way similar to \cite[Prop.\,8.5.7]{K-S90}.

\begin{properties}\label{Properties}\mbox{}
\begin{enumerate}
\item\label{Properties1}
An object $F$ of $\rD^\rb(p_X^{-1}\sho_{S})$ belongs to $\rD^\rb_\wcc(p_X^{-1}\sho_{S})$ if and only if it belongs to $\rD^\rb_\wcc(p_X^{-1}\C_{S})$.

\item\label{Properties2}
Remark \ref{R:0} applies to $\rD^\rb_\wcc(p_X^{-1}\sho_S)$ and $\rD^\rb_\cc(p_X^{-1}\sho_S)$.

\item\label{Properties3}
Proposition \ref{P:01} applies to $\rD^\rb_\wcc$.

\item\label{Properties4}
Propositions \ref{C:CCCC}, \ref{P:07}, and Corollary \ref{C:2.12} apply to $\rD^\rb_\cc(p_X^{-1}\sho_S)$.

\item\label{Properties5}
Corollary \ref{C:2.6.1} applies to $\rD^\rb_\wcc$, $\rD^\rb_\rc$ and $\rD^\rb_\cc$.
\end{enumerate}
\end{properties}

\subsection{Duality}
According to the syzygy theorem for the regular local ring $\sho_{S,s}$ (for any $s\in S$) and e.g. \cite[Prop.\,13.2.2(ii)]{K-S06} (for the opposite category), any object of $\rD^\rb_\coh(\sho_S)$ is locally quasi-isomorphic to a bounded complex of locally free $\sho_S$-modules of finite rank $L^\sbullet$. As a consequence, the local duality functor
\[
\bD: \rD^\rb_\coh(\sho_S)\to \rD^\rb_\coh(\sho_S),\quad \bD(\shf):=\Rhom_{\sho_S}(\shf,\sho_S)
\]
is seen to be an involution, i.e., the natural morphism $\id\to\bD\circ\bD$ is an isomorphism. However, the standard t-structure
\[
\big(\rD^{\rb,\leq0}_\coh(\sho_S),\rD^{\rb,\geq0}_\coh(\sho_S)\big)
\]
defined by $\shh^jG=0$ for $j>0$ (resp.~for $j<0$) is not interchanged by duality when $\dim S\geq1$ (see e.g., \cite[Prop.\,4.3]{Kashiwara04} in the algebraic setting). Nevertheless, we have:

\begin{lemma}\label{L:221}
Let $G$ be an object of $\rD^\rb_\coh(\sho_S)$. Assume that $\bD G$ belongs to $\rD^{\rb,\leq0}_\coh(\sho_S)$. Then $G$ belongs to $\rD^{\rb,\geq0}_\coh(\sho_S)$.
\end{lemma}

\begin{proof}
Setting $G'=\bD G$, the biduality isomorphism makes it equivalent to proving that $\bD G'$ belongs to $\rD^{\rb,\geq0}_\coh(\sho_S)$. The question is local on $S$ and we may therefore replace $G'$ with a bounded complex $L^\sbullet$ as above. Moreover, $L^\sbullet$ is quasi-isomorphic to such a bounded complex, still denoted by $L^\sbullet$, such that $L^k=0$ for $k>0$. Indeed, note first that the kernel $K$ of a surjective morphism of locally free $\sho_S$-modules of finite rank is also locally free of finite rank (being $\sho_S$-coherent and having all its germs $K_s$ free over $\sho_{S,s}$, because they are projective and $\sho_{S,s}$ is a regular local ring). By assumption, we have $\shh^j(L^\sbullet)=0$ for $j>0$. Let $k>0$ be such that $L^k\neq0$ and $L^\ell=0$ for $\ell>k$, and let $L^{\prime k-1}=\ker[L^{k-1}\to L^k]$. Then $L^\sbullet$ is quasi-isomorphic to $L^{\prime\sbullet}$ defined by $L^{\prime j}=L^j$ for $j<k-1$ and $L^{\prime j}=0$ for $j\geq k$. We conclude by induction on~$k$.

Now it is clear that $\bD G'\simeq \bD L^\sbullet$ is a bounded complex having terms in nonnegative degrees at most, and thus is an object of $\rD^{\rb,\geq0}_\coh(\sho_S)$.
\end{proof}

\begin{remark}
Let $G$ be an object of $\rD^\rb_\coh(\sho_S)$. Assume that $G$ and $\bD G$ belong to $\rD^{\rb,\leq0}_\coh(\sho_S)$. Then $G$ and $\bD G$ are $\sho_S$-coherent sheaves, hence $G$ and $\bD G$ are $\sho_S$-locally free.
\end{remark}

We now set $\omega_{X,S}=p_X^{-1}\sho_S[2\dim X]=p_X^!\sho_S$.

\begin{proposition}\label{P:12}
The functor $\bD:\rD^\rb(p_X^{-1}\sho_S)\to \rD^+(p_X^{-1}\sho_S)$ defined by $\bD F=\Rhom_{p_X^{-1}\sho_S}(F,\omega_{X,S})$ induces an involution $\rD^\rb_\rc(p_X^{-1}\sho_S)\to \rD^\rb_\rc(p_X^{-1}\sho_S)$ and $\rD^\rb_\cc(p_X^{-1}\sho_S)\to \rD^\rb_\cc(p_X^{-1}\sho_S)$.
\end{proposition}

We will also set $\bD'F=\Rhom_{p_X^{-1}\sho_S}(F,p_X^{-1}\sho_S)$.

\begin{proof}
Let us first show that, for $F$ in $\rD^\rb_\wrc(p_X^{-1}\sho_{S})$, the dual $\bD F$ also belongs to $\rD^\rb_\wrc(p_X^{-1}\sho_{S})$. let $(X_\alpha)$ be a $\mu$-stratification adapted to $F$. According to Corollary \ref{C:2.6}, it is enough to show that $i_\alpha^!\bD F$ has locally constant cohomology for each $\alpha$. One can use \cite[Prop.\,3.1.13]{K-S90} in our setting and get
\[
i_\alpha^!\bD F=\Rhom_{p_\alpha^{-1}\sho_S}(i_\alpha^{-1}F,\omega_{X_\alpha,S}).
\]
Locally on $X_\alpha$, $i_\alpha^{-1}F=p_\alpha^{-1}G$ for some $G$ in $\rD^\rb(\C_S)$ or $\rD^\rb(\sho_S)$. Then, locally on $X_\alpha$,
\begin{align*}
i_\alpha^!\bD F\simeq\Rhom_{p_\alpha^{-1}\sho_S}(p_\alpha^{-1}G,p_\alpha^!\sho_S)&=p_\alpha^!\Rhom_{\sho_S}(G,\sho_S)\\
&=p_\alpha^{-1}(\bD G)[2\dim X_\alpha].
\end{align*}
The proof for $F$ in $\rD^\rb_\wcc(p_X^{-1}\sho_{S})$ is similar. Moreover, by using Corollary \ref{C:2.12} instead of Corollary \ref{C:2.6} one shows that $\bD$ sends $\rD^\rb_\rc(p_X^{-1}\sho_{S})$ to itself and, according to Properties \ref{Properties}\eqref{Properties4}, $\rD^\rb_\cc(p_X^{-1}\sho_{S})$ to itself.

Let us prove the involution property. We have a natural morphism of functors $\id\to\bD\bD$. It is enough to prove the isomorphism property after applying $Li^*_{s_o}$ for each $s_o\in S$, according to Proposition \ref{P:91}. On the other hand, Proposition \ref{C:CC} implies that $Li^*_{s_o}$ commutes with $\bD$, so we are reduced to applying the involution property on $\rD^\rb_{\cc}(\C_X)$, according to the $\cc$-analogue of Proposition \ref{C:CCCC}, which is known to be true (see e.g.~\cite{K-S90}).
\end{proof}

\begin{remark}\label{R:222}
By using the biduality isomorphism and the isomorphism $i_x^!\bD F\simeq\bD i_x^{-1}F$ for $F$ in $\rD^\rb_\rc(p_X^{-1}\sho_S)$ or $\rD^\rb_\cc(p_X^{-1}\sho_S)$, where $i_x:\{x\}\times\nobreak S\hookrightarrow X\times S$ denotes the inclusion, we find a functorial isomorphism $i_x^{-1}\bD F\simeq\bD i_x^!F$.
\end{remark}

\subsection{Perversity}
We will now restrict to the case of $S$-$\C$-constructible complexes, which is the only case which will be of interest for us, although one could consider the case of $S$-$\R$-constructible complexes as in \cite[\S10.2]{K-S90}.

We define the category $\pD^{\leq0}_\cc(p_X^{-1}\sho_S)$ as the full subcategory of $\rD^\rb_\cc(p_X^{-1}\sho_S)$ whose objects are the $S$-$\C$-constructible bounded complexes~$F$ such that, for some adapted $\mu$-stratification $(X_\alpha)$ ($i_x$ is as above),
\begin{equation}\tag{Supp}\label{eq:support}
\forall\alpha,\;\forall x\in X_\alpha,\;\forall j>-\dim X_\alpha,\quad \shh^ji_x^{-1}F=0.
\end{equation}
Similarly, $\pD^{\geq0}_\cc(p_X^{-1}\sho_S)$ consists of objects $F$ such that
\begin{equation}\tag{Cosupp}\label{eq:cosupport}
\forall\alpha,\;\forall x\in X_\alpha,\;\forall j<\dim X_\alpha,\quad \shh^ji_x^!F=0.
\end{equation}

In the preceding situation in view of Corollary \ref{C:2.12} we have, similarly to \cite[Prop.10.2.4]{K-S90}:

\begin{lemma}\label{L:ana}\mbox{}
\begin{enumerate}
\item
$F\in \pD^{\leq0}_\cc(p_X^{-1}\sho_S)$ if and only if for any $\alpha$ and $j>-\dim(X_\alpha)$,
\[
\shh^j(i^{-1}_\alpha F)=0.
\]

\item
$F\in \pD^{\geq0}_\cc(p_X^{-1}\sho_S)$ if and only if for any $\alpha$ and $j<-\dim(X_\alpha)$,
\[
\shh^j(i^{!}_\alpha F)=0.
\]
\end{enumerate}
\end{lemma}

Namely, if $F\in \pD^{\leq0}_\cc(p_X^{-1}\sho_S)$ and $Z$ is a closed analytic subset of $X$ such that $\dim Z=k$, then $i_{Z\times S}^{-1} F$ is concentrated in degrees $\leq-k$,
and if $F'\in \pD^{\geq0}_\cc(p_X^{-1}\sho_S)$, then $i^!_{Z\times S} F'$ is concentrated in degrees $\geq-k$.
We have the following variant of \cite[Prop.10.2.7]{K-S90}:

\begin{proposition}\label{P:t29}
Let $F$ be an object of $\pD^{\leq0}_\wrc(p_X^{-1}\sho_S)$ and $F'$ an object of $\pD^{\geq0}_\wrc(p_X^{-1}\sho_S)$. Then
\[
\shh^j\Rhom_{p^{-1}_X\sho_S}(F,F')=0,\quad \text{for } j<0.
\]
\end{proposition}

\begin{proof}
Let $(X_\alpha)$ be a $\mu$-stratification of $X$ adapted to $F$ and $F'$. By assumption, for each $\alpha$, $i_\alpha^{-1}\shh^jF=\shh^ji_\alpha^{-1}F=0$ for $j>-\dim X_\alpha$. Similarly, $\shh^ji_\alpha^!F'=0$ for $j<-\dim X_\alpha$.

Let $X_\alpha$ be a stratum of maximal dimension such that
\[
i_\alpha^{-1}\shh^j\Rhom_{p^{-1}_X\sho_S}(F,F')\neq0\quad\text{for some $j<0$}.
\]
Let $V$ be an open neighbourhood of $X_\alpha$ in $X$ such that $V\setminus X_\alpha$ intersects only strata of dimension $>\dim X_\alpha$, and let $j_\alpha:(V\setminus X_\alpha)\times S\hookrightarrow V\times S$ be the inclusion. Then the complex $i_\alpha^{-1}Rj_{\alpha,*}j_\alpha^{-1}\Rhom_{p^{-1}_X\sho_S}(F,F')$ has nonzero cohomology in nonnegative degrees only: indeed, by the definition of $X_\alpha$, this property holds for $j_\alpha^{-1}\Rhom_{p^{-1}_X\sho_S}(F,F')$, hence it holds for $Rj_{\alpha,*}j_\alpha^{-1}\Rhom_{p^{-1}_X\sho_S}(F,F')$, and then clearly for the complex $i_\alpha^{-1}Rj_{\alpha,*}j_\alpha^{-1}\Rhom_{p^{-1}_X\sho_S}(F,F')$. From the distinguished triangle
\begin{multline*}
i_\alpha^!\Rhom_{p^{-1}_X\sho_S}(F,F')\to i_\alpha^{-1}\Rhom_{p^{-1}_X\sho_S}(F,F')\\
\to i_\alpha^{-1}Rj_{\alpha,*}j_\alpha^{-1}\Rhom_{p^{-1}_X\sho_S}(F,F')\xrightarrow{+1}
\end{multline*}
we conclude that $\shh^ji_\alpha^!\Rhom_{p^{-1}_X\sho_S}(F,F')\to\shh^j i_\alpha^{-1}\Rhom_{p^{-1}_X\sho_S}(F,F')= i_\alpha^{-1}\shh^j\Rhom_{p^{-1}_X\sho_S}(F,F')$ is an isomorphism for all $j<0$. Therefore, we obtain, for this stratum $X_\alpha$ and for any $j<0$,
\begin{align*}
i_\alpha^{-1}\shh^j\Rhom_{p^{-1}_X\sho_S}(F,F')&\simeq\shh^ji_\alpha^!\Rhom_{p^{-1}_X\sho_S}(F,F')\\
&\simeq\shh^j\Rhom_{p^{-1}_X\sho_S}(i_\alpha^{-1}F,i_\alpha^!F').
\end{align*}
Since $i_\alpha^{-1}F$ has nonzero cohomology in degrees $\leq-\dim X_\alpha$ at most and $i_\alpha^!F'$ in degrees $\geq-\dim X_\alpha$ at most, $\shh^j\Rhom_{p^{-1}_X\sho_S}(i_\alpha^{-1}F,i_\alpha^!F')=0$ for $j<0$, a contradiction with the definition of $X_\alpha$.
\end{proof}

\begin{theorem}\label{T:2.18}
$\pD^{\leq0}_\cc(p_X^{-1}\sho_S)$ and $\pD^{\geq0}_\cc(p_X^{-1}\sho_S)$ form a t-structure of $\rD^\rb_\cc(p_X^{-1}\sho_S)$, whose heart is denoted by $\mathrm{Perv}(p_X^{-1}\sho_S)$.
\end{theorem}

\begin{proof}[Sketch of proof]
We have to prove:
\begin{enumerate}
\item\label{T:2.181}
$\pD^{\leq0}_\cc\subset \pD^{\leq1}_\cc$ and
$\pD^{\geq0}_\cc\supset \pD^{\geq1}_\cc$.

\item\label{T:2.182}
For $F\in \pD^{\leq0}_\cc(p_X^{-1}\sho_S)$ and $F'\in \pD^{\geq1}_\cc(p_X^{-1}\sho_S)$,
\[
\Hom_{\rD^\rb(p_X^{-1}\sho_S)}(F,F')=0.
\]

\item\label{T:2.183}
For any $F\in \rD^\rb_\cc(p_X^{-1}\sho_S)$ there exist $F'\in \pD^{\leq0}_\cc(p_X^{-1}\sho_S)$ and $F''\in \pD^{\geq1}_\cc(p_X^{-1}\sho_S)$, giving rise to a distinguished triangle
$F'\to F\to F''\overset{+1}{\to}$.
\end{enumerate}

Then, following the line of the proof of \cite[Theorem 10.2.8]{K-S90}, we observe that \eqref{T:2.181} is obvious and \eqref{T:2.182} follows from Proposition \ref{P:t29}. Now, \eqref{T:2.183} is deduced by mimicking stepwise the proof of (c) in \cite[Theorem 10.2.8]{K-S90}.\end{proof}

According to the preliminary remarks before Lemma \ref{L:221}, one cannot expect that the previous t-structure is interchanged by duality when $\dim S\geq\nobreak1$. However we have:

\begin{proposition}\label{P:227}
Let $F$ be an object of $\pD^{\leq0}_\cc(p_X^{-1}\sho_S)$ such that $\bD F$ also belongs to $\pD^{\leq0}_\cc(p_X^{-1}\sho_S)$. Then $F$ and $\bD F$ are objects of $\mathrm{Perv}(p_X^{-1}\sho_S)$.
\end{proposition}

\begin{proof}
Let us fix $x\!\in X_\alpha$. We have $i_x^!F\simeq \bD(i_x^{-1}\bD F)$, as already observed in Remark \ref{R:222}. By assumption $G:=i_x^{-1}\bD F$ belongs to $\rD^{\rb,\leq-\dim X_\alpha}_\coh(\sho_S)$, and Lemma \ref{L:221} suitably shifted and applied to $\bD G$ implies that $\bD G$ belongs to $\rD^{\rb,\geq\dim X_\alpha}_\coh(\sho_S)$, which is the cosupport condition (Cosupp) for~$F$.
\end{proof}

Assume $F\in \mathrm{Perv}(p_X^{-1}\sho_S)$. The description of the dual standard t\nobreakdash-structure on $\rD^\rb_\coh(\sho_S)$ given in \cite[\S4]{Kashiwara04} supplies the following refinement to \eqref{eq:support} and \eqref{eq:cosupport} when $\bD F$ is also perverse.

\begin{corollary}\label{C:tmf2}
Let $F\in \mathrm{Perv}(p_X^{-1}\sho_S)$ and assume that $\bD F\in \mathrm{Perv}(p_X^{-1}\sho_S)$. Let $(X_\alpha)$ be a stratification adapted to $F$. Then for each $\alpha$, each $x\in X_\alpha$ and each closed analytic subset $Z\subset S$, we have
\begin{equation}\tag{Cosupp$+$}\label{eq:cosupportplus}
\shh^k(i^!_{Z\times\{x\}}F)=0,\quad\forall k<\codim_SZ+\dim X_\alpha.
\end{equation}
\end{corollary}
(The perversity of $F$ only gives the previous property when $Z=S$.) 

\section{The de~Rham complex of a holonomic $\shd_{X\times S/S}$-module}

In what follows $X$ and $S$ denote complex manifolds and we set $n=\dim X$, $\ell=\dim S$. We shall keep the notation of the preceding section. Let $\pi:T^*(X\times S)\to T^*X\times S$ denote the projection and let $\shd_{X\times S/S}$ denote the subsheaf of $\shd_{X\times S}$ of relative differential operators with respect to $p_X$ (see \cite[\S2.1 \& 2.2]{Sch-Sch94}).

Recall that $p^{-1}_X\sho_{S}$ is contained in the center of $\shd_{X\times S/S}$. With the same proof as for Proposition \ref{C:CC} we obtain:

\begin{proposition}\label{P:5}
Let $s_o\in S$ be given. Let $\shm$ and $\shn$ be objects of $\rD^\rb(\shd_{X\times S/S})$. Then, there is a well-defined natural morphism
\[
Li^*_{s_o}(\Rhom_{\shd_{X\times S/S}}(\shm, \shn))\to \Rhom_{i^*_{s_o}(\shd_{X\times S/S})}(Li^*_{s_o}(\shm), Li^*_{s_o}(\shn))
\]
which is an isomorphism in $\rD^\rb(\C_X)$.
\end{proposition}

\subsection{Duality for coherent $\shd_{X\times S/S}$-modules}
We refer for instance to \cite[Appendix]{Kashiwara03} for the coherence properties of the ring $\shd_{X\times S/S}$. The classical methods used in the absolute case, i.e, for coherent $\mathcal{D}_X$-objects (see for instance \cite[Prop. 2.1.16]{Mebkhout89}, \cite[Prop. 2.7-3]{Mebkhout04}) apply here:

\begin{proposition}\label{P:4}
Let $\shm$ be a coherent $\shd_{X\times S/S}$-module. Then $\shm$ locally admits a resolution of length at most $2n+\ell$ by free $\shd_{X\times S/S}$-modules of finite rank.
\end{proposition}

Proposition \ref{P:4} and \cite[Prop.\,13.2.2(ii)]{K-S06} (for the opposite category) imply:
\begin{corollary}\label{C:8}
Let $\shm\in \rD^\rb_\coh(\shd_{X\times S/S})$. Let us assume that $\shm$ is concentrated in degrees $[a,b]$. Then, in a neighborhhod of each $(x,z)\in X\times S$, there exist a complex $\shl^\sbullet$ of free $\shd_{X\times S/S}$-modules of finite rank concentrated in degrees $[a-2n-\ell, b]$ and a quasi-isomorphism $ \shl^\sbullet\to \shm$.
\end{corollary}

We set $\Omega_{X\times S/S}=\Omega^{n}_{X\times S/S}$, where $\Omega^{n}_{X\times S/S}$ denotes the sheaf of relative differential forms of degree $n=\dim X$.

\begin{definition}\label{D:23} The duality functor
$\bD(\cdot): \rD^\rb(\shd_{X\times S/S}) \to \rD^\rb(\shd_{X\times S/S})$ is defined as:
\[
\shm\mapsto\bD\shm=\Rhom_{\shd_{X\times S/S}}(\shm, \shd_{X\times S/S}\otimes_{\sho_{X\times S}}\Omega_{X\times S/S}^{\otimes^{-1}})[n].
\]
We also set $\bD'\shm:=R\shh om_{\shd_{X\times S/S}}(\shm, \shd_{X\times S/S})\in \rD^{\rb}(\shd_{X\times S/S}^\opp)$.
\end{definition}

By Proposition \ref{P:4}, $\shd_{X\times S/S}$ has finite cohomological dimension, so \cite[(A.11)]{Kashiwara03} gives a natural morphism in $\rD^\rb(\shd_{X\times S/S})$:
\begin{equation}\label{E:D}
\shm\to \bD'\bD'\shm\simeq\bD\bD\shm.
\end{equation}
Moreover, in view of Corollary \ref{C:8}, if $\shm\in \rD^\rb_\coh(\shd_{X\times S/S})$, then $\bD'\shm \in \rD^\rb_\coh(\shd_{X\times S/S}^\opp).$
Indeed, we may choose a local free finite resolution $\shl^\sbullet$ of~$\shm$, so that $\bD'\shm$ is quasi isomorphic to the transposed complex $(\shl^\sbullet)^t$ whose entries are free.

By the same argument we deduce that \eqref{E:D} is an isomorphism whenever $\shm\in \rD^\rb_\coh(\shd_{X\times S/S})$.

Again by Proposition \ref{P:4}, $\shd_{X\times S/S}$ has finite flat dimension so we are in conditions to apply \cite[(A.10)]{Kashiwara03}: given $\shm, \shn \in \rD^\rb(\shd_{X\times S/S})$ there is a natural morphism:
\begin{equation}\label{E:DDD}
\bD'\shm \overset{L}{\otimes}_{\shd_{X\times S/S}}\shn \to \Rhom_{\shd_{X\times S/S}}(\shm,\shn)
\end{equation} which an isomorphism provided that $\shm$ or $\shn$ belong to $\rD^\rb_\coh(\shd_{X\times S/S})$.
When $\shm, \shn \in \rD^\rb_\coh(\shd_{X\times S/S})$, composing \eqref{E:DDD} with the biduality isomorphism \eqref{E:D} gives a natural isomorphism
\begin{equation}\label{E:DD}
\Rhom_{\shd_{X\times S/S}}(\shm,\shn)\simeq \Rhom_{\shd_{X\times S/S}}(\bD\shn, \bD\shm).
\end{equation}

\subsection{Characteristic variety}
Recall (see \cite[\S III.1.3]{Schapira85}) that the characteristic variety $\Char\shm$ of a coherent $\shd_{X\times S/S}$-module $\shm$ is the support in $T^*X\times S$ of its graded module with respect to any (local) good filtration. One has (see \cite[Prop.\,III.1.3.2]{Schapira85})
\begin{equation}\label{E:CHAR}
\begin{split}
\Char(\shd_{X\times S}\otimes_{\shd_{X\times S/S}} \shm)&=\pi^{-1}\Char\shm,\\
\Char\shm&=\pi\big(\Char(\shd_{X\times S}\otimes_{\shd_{X\times S/S}} \shm)\big).
\end{split}
\end{equation}

One may as well define the characteristic variety of an object $\shm\in \rD^\rb_\coh(\shd_{X\times S/S})$ as the union of the characteristic varieties of its cohomology modules. By the flatness of $\shd_{X\times S}$ over $\shd_{X\times S/S}$, \eqref{E:CHAR} holds for any object of $\rD^\rb_\coh(\shd_{X\times S/S})$.

\begin{proposition}[{\cite[Prop.\,2.5]{Sch-Sch94}}]\label{E:D5}
For $\shm\in \rD^\rb_\coh(\shd_{X\times S/S})$ we have
\[
\Char(\shm)=\Char (\bD \shm).
\]
\end{proposition}

\subsection{The de Rham and solution complexes}
For an object $\shm$ of $\rD^\rb(\shd_{X\times S/S})$ we define the functors
\begin{align*}
\DR\shm&:=\Rhom_{\shd_{X\times S/S}}(\sho_{X\times S},\shm),\\
\Sol\shm&:=\Rhom_{\shd_{X\times S/S}}(\shm,\sho_{X\times S})
\end{align*}
which take values in $\rD^{\rb}(p_X^{-1}\sho_S)$. If $\shm$ is a $\shd_{X\times S/S}$-module, that is, a $\sho_{X\times S}$-module equipped with an integrable relative connection $\nabla:\shm\to\Omega^1_{X\times S/S}\otimes\nobreak\shm$, the object $\DR\shm$ is represented by the complex $(\Omega^\sbullet_{X\times S/S}\otimes_{\sho_{X\times S}}\shm,\nabla)$.

Noting that $\Rhom_{\shd_{X\times S/S}}(\sho_{X\times S},\shd_{X\times S/S})\simeq \Omega_{X\times S/S}[-\dim X]$ we get
\[
\bD\sho_{X\times S}\simeq \sho_{X\times S}.
\]
For $\shn=\sho_{X\times S}$, \eqref{E:DD} implies a natural isomorphism, for $\shm\in \rD^\rb_\coh(\shd_{X\times S/S})$:
\begin{equation}\label{E:D3}
\Sol\shm\simeq \DR\bD\shm.
\end{equation}

\subsection{Holonomic $\shd_{X\times S/S}$-modules}
Let $\shm$ be a coherent $\shd_{X\times S/S}$-module. We say that it is \emph{holonomic} if its characteristic variety $\Char\shm\subset T^*X\times S$ is contained in $\Lambda\times S$ for some closed conic Lagrangian complex analytic subset of $T^*X$. We will say that a complex $\mu$-stratification $(X_\alpha)$ is adapted to $\shm$ if $\Lambda\subset\bigcup_\alpha T^*_{X_\alpha}X$. Similar definitions hold for objects of $\rD^\rb_\hol(\shd_{X\times S/S})$.

An object $\shm\in \rD^\rb_\coh(\shd_{X\times S/S})$ is said to be holonomic if its cohomology modules are holonomic. We denote the full triangulated category of holonomic complexes by $\rD^\rb_\hol(\shd_{X\times S/S})$.

\begin{corollary}[of Prop.\,\ref{E:D5}]\label{Cor:C3.6}
If $\shm$ is an object of $\rD^\rb_\hol(\shd_{X\times S/S})$, then so is $\bD\shm$.
\end{corollary}

\begin{theorem}\label{T:1}
Let $\shm$ be an object of $\rD^\rb_\hol(\shd_{X\times S/S})$. Then $\DR(\shm)$ and $\Sol\shm$ belong to $\rD^\rb_\cc(p^{-1}_X\sho_S)$.
\end{theorem}

\begin{proof}
Firstly, it follows \cite[Prop.\,11.3.3]{K-S90}, that $\Sol(\shm)$ and $\DR(\shm)$ have their micro-support contained in $\Lambda\times T^*S$ (see \cite[p.\,11 \& Th.\,2.13]{Sch-Sch94}) and, according to Proposition \ref{P:6}, these complexes are objects of $\rD^\rb_\wcc(p^{-1}_X\sho_S)$.

Let $x\in X$. In order to prove that $i_x^{-1}\DR\shm$ has $\sho_S$-coherent cohomology, we can assume that $x$ is a stratum of a stratification adapted to $\DR\shm$ and we use Lemma \ref{L:2.7} to get $i_x^{-1}\DR\shm\simeq Rp_{\ov\epsilon,*}(\C_{\ov B_\epsilon\times S}\otimes_\C\DR\shm)$ for $\epsilon$ small enough, where $\ov B_\epsilon$ is a closed ball of radius $\epsilon$ centered at $x$. One then remarks that $(\C_{\ov B_\epsilon\times S},\shm)$ forms a relative elliptic pair in the sense of \cite{Sch-Sch94}, and Proposition 4.1 of loc.\,cit.\ gives the desired coherence.

The statement for $\Sol\shm$ is proved similarly.
\end{proof}

\begin{lemma}[{see \cite[Prop.\,1.2.5]{Sabbah05}}]\label{L:istar}
For $\shm$ in $\rD^\rb_\hol(\shd_{X\times S/S})$ with adapted stratification $(X_\alpha)$ and for any $s_o\in S$, $Li^*_{s_o}\shm$ is $\shd_X$-holonomic and $(X_\alpha)$ is adapted to it.
\end{lemma}

\begin{corollary}\label{C:H1}
For $\shm\in \rD^\rb_\hol(\shd_{X\times S/S})$, there is a natural isomorphism $\bD'\Sol \shm\simeq\DR\shm$.
\end{corollary}
\begin{proof}
We consider the canonical pairing
\[
\DR\shm\overset{L}{\otimes}_{p^{-1}_X\sho_S}\Sol \shm\to p_X^{-1}\sho_S
\]
which gives a natural morphism
\[
\DR\shm\to\bD'\Sol \shm
\]
in $\rD^\rb_{\cc}(p_X^{-1}\sho_S)$.
We have for each $s_o\in S$, by Proposition \ref{P:5}
\begin{align*}
Li^*_{s_o}(\DR\shm)&\simeq \DR Li^*_{s_o}(\shm),\\
Li^*_{s_o}(\Sol \shm)&\simeq \Sol Li^*_{s_o}(\shm) .
\end{align*}

Since $Li^*_{s_o}(\shm)\in \rD^\rb_\hol(\shd_X)$ by Lemma \ref{L:istar}, we have
\[
\DR Li^*_{s_o}(\shm)\simeq \bD'\Sol Li^*_{s_o}(\shm),
\]
so by Proposition \ref{P:5} and Proposition \ref{C:CC}
\[
\bD'\Sol Li^*_{s_o}(\shm)\simeq \bD'Li^*_{s_o}(\Sol \shm)\simeq Li^*_{s_o}(\bD'\Sol \shm).
\]
The assertion then follows by Proposition \ref{P:91}.
\end{proof}

In the following proposition, the main argument is that of strictness, which is essential. We will set $\pDR\shm:=\DR\shm[\dim X]$ and $\pSol\shm=\Sol\shm[\dim X]$.

\begin{proposition}\label{P:3.3}
Let $\shm$ be a holonomic $\shd_{X\times S/S}$-module which is strict, i.e., which is $p^{-1}\sho_S$-flat. Then $\pDR\shm$ satisfies the support condition \eqref{eq:support} with respect to a $\mu$-stratification adapted to~$\shm$.
\end{proposition}

\begin{proof}
We prove the result by induction on $\dim S$. Since it is local on $S$, we consider a local coordinate $s$ on $S$ and we set $S'=\{s=0\}$. The strictness property implies that we have an exact sequence
\[
0\to\shm\xrightarrow{~s~}\shm\to i_{S'}^*\shm\to0,
\]
and $i_{S'}^*\shm$ is $\shd_{X\times S'/S'}$-holonomic and $p^{-1}\sho_{S'}$-flat. We deduce an exact sequence of complexes $0\to\pDR\shm\xrightarrow{s}\pDR\shm\to\pDR i_{S'}^*\shm\to0$.

Let $X_\alpha$ be a stratum of a $\mu$-stratification of $X$ adapted to $\shm$ (hence to $i_{S'}^*\shm$, after Lemma \ref{L:istar}). For $x\in X_\alpha$, let $k$ be the maximum of the indices~$j$ such that $\shh^ji_x^{-1}\pDR\shm\neq0$. For any $S'$ as above, we have a long exact sequence
\[
\cdots\to\shh^ki_x^{-1}\pDR\shm\xrightarrow{~s~}\shh^ki_x^{-1}\pDR\shm\to\shh^ki_x^{-1}\pDR i_{S'}^*\shm\to0.
\]
If $k>-\dim X_\alpha$, we have $\shh^ki_x^{-1}\pDR i_{S'}^*\shm=0$, according to the support condition for $i_{S'}^*\shm$ (inductive assumption), since $(X_\alpha)$ is adapted to it. Therefore, $s:\shh^ki_x^{-1}\pDR\shm\to\shh^ki_x^{-1}\pDR\shm$ is onto. On the other hand, by Theorem~\ref{T:1}, $\shh^ki_x^{-1}\pDR\shm$ is $\sho_S$-coherent. Then Nakayama's lemma implies that $\shh^ki_x^{-1}\pDR\shm=0$ in some neighbourhood of $S'$. Since~$S'$ was arbitrary, this holds all over $S$, hence the assertion.
\end{proof}

\begin{proof}[Proof of Theorem \ref{th:main2}]
It is a direct consequence of the following.

\begin{theorem}\label{T:3.4}
Let $\shm$ be an object of $\rD^\rb_\hol(\shd_{X\times S/S})$ and let $\bD\shm$ be the dual object. Then there is an isomorphism $\pDR\bD\shm\simeq\bD\pDR\shm$.
\end{theorem}

Indeed, with the assumptions of Theorem \ref{th:main2}, $\bD\shm$ is holonomic since~$\shm$ is so (see Corollary \ref{Cor:C3.6}), and both $\shm$ and $\bD\shm$ are strict. Then both $\pDR\shm$ and $\pDR\bD\shm$ satisfy the support condition, according to Proposition \ref{P:3.3}. Hence, according to Theorem \ref{T:3.4} and Proposition \ref{P:227}, $\pDR\shm$ satisfies the cosupport condition.

Similarly, $\pSol\shm\simeq\bD\pDR\shm$ and $\bD(\pSol\shm)\simeq\pDR\shm$ both satisfy the support condition, hence $\Sol\shm[\dim X]$ is a perverse object.
\end{proof}

\begin{proof}[Proof of Theorem \ref{T:3.4}]
Combining \eqref{E:DD} with \cite[Ex.\,II.24\,(iv)]{K-S90} (with $f=\id$, $\sha=\shd_{X\times S/S}$ and $\shb=p_X^{-1}\sho_S$) entails, for any $\shn\in \rD^\rb_\coh(\shd_{X\times S/S})$, a natural morphism $$\Rhom_{\shd_{X\times S/S}}(\shn,\shm)\to \Rhom_{p_X^{-1}\sho_S}(\DR\bD\shm, \DR\bD\shn).$$
When $\shn=\sho_{X\times S}$, we obtain a natural morphism
$$ \DR\shm\to \bD'\DR\bD\shm,\quad\text{that is,}\quad \pDR\shm\to \bD\pDR\bD\shm.$$
Suppose now that $\shm\in \rD^\rb_\hol(\shd_{X\times S/S})$. Recall that $\bD\shm\in \rD^\rb_\hol(\shd_{X\times S/S})$, so $\pDR\bD\shm\in \rD^\rb_\cc(p^{-1}_X\sho_S)$.

Hence, by biduality, we get a morphism
\begin{equation}\label{E:i}
\bD\pDR\shm\leftarrow\pDR\bD\shm.
\end{equation}

On the other hand, since $Li^*_{s_o}(\shm)\in \rD^\rb_\hol(\shd_X)$ for each $s_o\in S$, the morphisms above induce isomorphisms
$$Li^*_{s_o}(\bD\pDR\shm)\simeq \pDR\bD Li^*_{s_o}(\shm)$$
according to Proposition \ref{C:CC} and Proposition \ref{P:5}, where in the right hand side we consider the duality for holonomic $\shd_X$-modules. Thus \eqref{E:i} is an isomorphism by Proposition \ref{P:91} and the local duality theorem for holonomic $\shd_X$-modules (see \cite{Narvaez04} and the references given there).
\end{proof}

\begin{example}
Let $X$ be the open unit disc in $\C$ with coordinate $x$ and let~$S$ be a connected open set of $\C$ with coordinate $s$. Let $\varphi:S\to\nobreak\C$ be a non constant holomorphic function on $S$ and consider the holonomic $\shd_{X\times S/S}$-module $\shm=\shd_{X\times S/S}/\shd_{X\times S/S}\cdot P$, with $P=x\partial_x-\varphi(s)$. It is easy to check that $\shm$ has no $\sho_S$-torsion and admits the resolution $0\to\nobreak\shd_{X\times S/S}\xrightarrow{{}\cdot P}\nobreak\shd_{X\times S/S}\to\nobreak\shm\to\nobreak0$, so that the dual module $\bD\shm$ has a similar presentation and is also $\sho_S$-flat. The complex $\pSol\shm$ is represented by $0\to\sho_{X\times S}\xrightarrow{P\cdot{}}\sho_{X\times S}\to0$ (terms in degrees $-1$ and $0$). Consider the stratification $X_1=X\setminus\{0\}$ and $X_0=\{0\}$ of $X$. Then $\shh^{-1}\pSol\shm_{|X_1}$ is a locally constant sheaf of free $p^{-1}_X\sho_S$-modules generated by a local determination of~$x^{\varphi(s)}$, and $\shh^0\pSol\shm_{|X_1}=0$. On the other hand, $\shh^{-1}\pSol\shm_{|X_0}=0$ and $\shh^0\pSol\shm_{|X_0}$ is a skyscraper sheaf on $X_0\times S$ supported on $\{s\in S\mid\varphi(s)\in\Z\}$.

For each $x_0$ we have
\begin{multline*}
{{i_{x_0}^{!}}} (\pSol\shm)\\
\simeq i^{-1}_{\{x_0\}\times S}\Rhom_{\shd_{X\times S}}(\shd_{X\times S}\otimes_{\shd_{X\times S/S}}\shm, R\Gamma_{\{x_0\}\times S|X\times S}\sho_{X\times S})[\dim X]\\
\simeq i^{-1}_{\{x_0\}\times S}\Rhom_{\shd_{X\times S}}(\shd_{X\times S}\otimes_{\shd_{X\times S/S}}\shm, B_{\{x_0\}\times S|X\times S})
\end{multline*}
where $B_{\{x_0\}\times S|X\times S}:=\shh^1_{[\{x_0\}\times S]}(\sho_{X\times S})$ denotes the sheaf of holomorphic hyperfunctions (of finite order) along $x=x_0$ (cf.~\cite{S-K-K73}). The second isomorphism follows from the fact that $\shd_{X\times S}\otimes_{\shd_{X\times S/S}}\shm$ is regular specializable along the submanifold $x=x_0$ (cf.~\cite{L-MF88}). 

Recall that the sheaves $B_{\{x_0\}\times S|X\times S}$ are flat over $p_X^{-1}\sho_S$ because locally they are inductive limits of free $p_X^{-1}\sho_S$-modules of finite rank.
 
Since ${i_{x_0}^{!}} (\pSol\shm)$ is quasi isomorphic to the complex 
$$0\to B_{\{x_0\}\times S|X\times S}|_{\{x_0\}\times S}\xrightarrow{~P\cdot{}~}B_{\{x_0\}\times S|X\times S}|_{\{x_0\}\times S}\to 0$$ it follows that the flat dimension over $\sho_S$ of ${{i_{x_0}^{!}}} (\pSol\shm)$ in the sense of \cite[\S4]{Kashiwara04} is $\leq 0$ for any $x_0$. Moreover, $\shh^0{i_{x_0}^{!}} (\pSol\shm)=0$ and, if $x_0\neq\nobreak0$, $\shh^1{i_{x_0}^{!}} (\pSol\shm)$ is a locally free $\sho_S$-module of rank $1$. Hence the flat dimension of ${i_{x_0}^{!}} (\pSol\shm)$ is $\leq 1$. This shows explicitly that $\pSol\shm$ satisfies the condition \eqref{eq:cosupportplus} of Corollary \ref {C:tmf2}.
\end{example}

\section{Application to mixed twistor $\shd$-modules}
Let $\shr_{X\times\C}$ be the sheaf on $X\times\C$ of $z$-differential operators, locally generated by $\sho_{X\times\C}$ and the $z$-vector fields $z\partial_{x_i}$ in local coordinates $(x_1,\dots,x_n)$ on~$X$. When restricted to $X\times\C^*$, the sheaf $\shr_{X\times\C^*}$ is isomorphic to $\shd_{X\times\C^*/\C^*}$.

A mixed twistor $\shd$-module on $X$ (see \cite{Mochizuki11}) is a triple $\sht=(\shm',\shm'',C)$, where $\shm',\shm''$ are holonomic $\shr_{X\times\C}$-modules and $C$ is a certain pairing with values in distributions, that we will not need to make precise here. Such a triple is subject to various conditions. We say that a $\shd_{X\times\C^*/\C^*}$-module $\shm$ underlies a mixed twistor $\shd$-module $\sht$ if $\shm$ is the restriction to $X\times\C^*$ of~$\shm'$ or $\shm''$.

Theorem \ref{th:main} is now a direct consequence of the following properties of mixed twistor $\shd$-modules, since they imply that $\shm$ satisfies the assumptions of Theorem \ref{th:main2}. If $\shm$ underlies a mixed twistor $\shd$-module, then
\begin{itemize}
\item
there exists a locally finite filtration $W_\sbullet\shm$ indexed by $\Z$ by $\shr_{X\times\C}$-submodules such that each graded module underlies a pure polarizable twistor $\shd$-module; then each $\mathrm{gr}_\ell^W\shm$ is strict and holonomic (see \cite[Prop.\,4.1.3]{Sabbah05} and \cite[\S17.1.1]{Mochizuki08}), and thus so is $\shm$;
\item
the dual of $\shm$ as a $\shr_{X\times\C^*}$-module also underlies a mixed twistor $\shd$-module, hence is also strict holonomic (see \cite[Th.\,12.9]{Mochizuki11}); using the isomorphism $\shr_{X\times\C^*}\simeq\shd_{X\times\C^*/\C^*}$, we see that the dual $\bD\shm$ as a $\shd_{X\times\C^*/\C^*}$-module is strict and holonomic.\qed
\end{itemize}

\newcommand{\SortNoop}[1]{}\def\cprime{$'$}
\providecommand{\bysame}{\leavevmode\hbox to3em{\hrulefill}\thinspace}
\providecommand{\MR}{\relax\ifhmode\unskip\space\fi MR }
\providecommand{\MRhref}[2]{%
  \href{http://www.ams.org/mathscinet-getitem?mr=#1}{#2}
}
\providecommand{\href}[2]{#2}

\end{document}